\documentclass{amsart}[11pt]

\usepackage[all]{xy}
\usepackage{amsmath}
\usepackage{amssymb}
\usepackage{MnSymbol}
\usepackage{mathrsfs}

\newtheorem{theorem}{Theorem}[section]
\newtheorem{lemma}[theorem]{Lemma}

\newtheorem{proposition}[theorem]{Proposition}
\newtheorem{question}[theorem]{Question}

\newtheorem{manualtheoreminner}{Theorem}
\newenvironment{manualtheorem}[1]{%
  \renewcommand\themanualtheoreminner{#1}%
  \manualtheoreminner
}{\endmanualtheoreminner}

\theoremstyle{definition}
\newtheorem{definition}[theorem]{Definition}

\theoremstyle{remark}
\newtheorem{remark}[theorem]{Remark}

\usepackage{graphicx}
\graphicspath{ {./images/} }

\usepackage{pst-node}
\usepackage{tikz-cd} 

\numberwithin{equation}{section}

\begin{document}
	
\title[Quadratic isoperimetric inequality]{quadratic isoperimetric inequality for mapping tori of hyperbolic groups}

\author{Qianwen Sun}
\address{Department of Mathematics,
        University of Michigan,
        Ann Arbor, MI 48109.
	 }
\email{qwsun@umich.edu}

\date{}
\maketitle

\textbf{Abstract}:  In this paper, we establish that the mapping torus of a one-ended torsion-free hyperbolic group exhibits a quadratic isoperimetric inequality. 

\section{Introduction}

Dehn functions have garnered significant attention in recent years. Coined after Max Dehn, a Dehn function is an optimal function that bounds the area of a word, representing the identity, in terms of the defining relators within a finitely presented group. It is closely tied to the algorithmic complexity of the word problem, that is determining whether a given word equals 1. A finitely presented group has a solvable word problem if and only if it possesses a recursive Dehn function \cite{gersten1993isoperimetric}. An important milestone connecting Dehn functions and hyperbolic groups, established by Gromov \cite{gromov1987hyperbolic}, states that a finitely presented group is hyperbolic if and only if its Dehn function is linear, or equivalently, satisfies a linear isoperimetric inequality.

Gromov also highlighted an isoperimetric gap in the same publication: if a finitely presented group satisfies a subquadratic isoperimetric inequality, then its Dehn function is linear. Notably, no group has Dehn function $n^d$, where $d\in (1,2)$.

Mapping tori have proven to be a powerful tool in low-dimensional topology. As any outer automorphism $[f]$ of a surface group is realized by a homeomorphism $f$ of the underlying surface $S$, the mapping torus represents the fundamental group of a 3-manifold. An influential result by Thurston \cite{thurston1982three} asserts that, in this scenario, the 3-manifold $M_f$ is hyperbolic if and only if $f$ is a pseudo-Anosov homeomorphism of $S$. In all other cases, $\pi_1(M_f)$ has a quadratic Dehn function because it is automatic \cite{epstein1992word}.

Building upon Sela's work, it was established that if a mapping torus of a torsion-free hyperbolic group $G$ is hyperbolic, then $G$ must be a free product of free groups and surface groups \cite{brinkmann2000hyperbolic}. Consequently, in the case of one-ended groups, G becomes a surface group, and the automorphism becomes a pseudo-Anosov automorphism. However, if the mapping torus is not hyperbolic, the nature of its Dehn function becomes a matter of inquiry, leading us to our primary problem.

\begin{question}
For an automorphism $\varphi$ of a hyperbolic group $G$, does the mapping torus $M_{\varphi}$ always have a quadratic isoperimetric inequality? 
\end{question}

Brinkmann \cite{brinkmann2000hyperbolic} classified all hyperbolic mapping tori of free groups. In 2000, Macura \cite{macura2000quadratic} demonstrated that the mapping torus of a polynomially growing automorphism of a finitely generated free group satisfies a quadratic isoperimetric inequality.
Bridson and Groves \cite{bridson2010quadratic} extended this result to arbitrary automorphisms in 2010. 

In this paper, we establish the following result.

\begin{theorem}\label{main}
Let $G$ be a one-ended torsion-free hyperbolic group and $\varphi$ be an automorphism of $G$. Then the corresponding mapping torus $M_{\varphi}$ satisfies a quadratic isoperimetric inequality.
\end{theorem}

The strategy of proof relies on two important results. First, generalized the work of Rips and Sela \cite{RS}, Levitt \cite{levitt2005automorphisms} gives a description, up to finite index, of the group of automorphisms of $G$, based on the JSJ decomposition. There is a clear picture of polynomially-growing automorphisms, which are multi-twists up to a power. Second, the recent work of Dahmani and Krishna \cite{dahmani2020relative} gives a relatively hyperbolic structure for the mapping torus. Given the work of Farb \cite{farb1998relatively} on the Dehn function of relatively hyperbolic groups, it is enough to consider the mapping torus of multi-Dehn twists. We establish the case of the Dehn twist by giving a standard form of the Van-Kampen diagram and work from there.      

This paper is organized as follows: first, in Section \ref{preliminary}, we provide a review of the necessary preliminary concepts. Section~\ref{oper: Van-Kampen} introduces a set of operations on the Van-Kampen diagram. These operations are employed to separate and estimate the number of primary cells and $t$-cells. Section~\ref{oper: graph of groups} presents a standardized form of Dehn twists achieved through specific operations on the graph of groups.
Finally, in Section~\ref{proof}, we combine the aforementioned results to present the proof of Theorem~\ref{main}.

\textbf{Acknowledgment}
The majority of this work was done at the University of Illinois Chicago when I was pursuing my Ph.D. degree. I would like to thank my advisor, Daniel Groves, who proposed this topic and provided many valuable ideas and suggestions. I would also like to thank the anonymous referees of an earlier draft. They drew my attention to relative Dehn functions, which turned out to be a central tool in solving the general case.

\section{Preliminaries}\label{preliminary}
In this section, we review some preliminary concepts in geometric group theory. We recommend consulting standard references such as \cite{MA}, \cite{R} for further exploration.

\subsection{(Relatively) Hyperbolic groups}

The Cayley graph, originally introduced by Cayley's theorem \cite{C}, plays a significant role in combinatorial and geometric group theory. Let us begin by recalling its definition.

\begin{definition}
Given a generating set $S$ of a group $G$, the \emph{Cayley graph} of $G$ with respect to $S$, denoted by $\Gamma(G, S)$, is a directed labeled graph. The vertices correspond to elements of $G$, and there is an edge from $g$ to $h$ if and only if $h=gs$ for some $s\in S$. The edge is labeled by $s$.
\end{definition}

\begin{definition}
A group is \emph{$\delta$-hyperbolic} if every geodesic triangle in its Cayley graph is $\delta$-slim. This means that each side of the triangle is contained within the $\delta$-neighborhood of the union of the other two sides. If a group is $\delta$-hyperbolic for some $\delta$, it is referred to as a \emph{hyperbolic group}.
\end{definition}

A generalization of hyperbolic groups is the concept of relatively hyperbolic groups. There are multiple equivalent definitions available; for detailed discussions, refer to \cite{farb1998relatively}, \cite{BD}, \cite{O}. Here, we define it in terms of the coned-off Cayley graph, as proposed by Farb \cite{farb1998relatively}.

\begin{definition}
Let $\mathscr{H}=\{H_i\}$ be a family of finitely generated subgroups of a finitely generated group $G$, and $\Gamma$ be a Cayley graph of $G$. The coned-off Cayley graph of $G$ with respect to $\mathscr{H}$, denoted by $\hat{\Gamma}=\hat{\Gamma}(\mathscr{H})$, is obtained by adding a vertex $v(gH_i)$ for each left coset $gH_i$, $H_i\in \mathscr{H}$. Additionally, an edge of length $1/2$ is added from each vertex corresponding to elements in $gH_i$ to $v(gH_i)$, $H_i\in \mathscr{H}$.
\end{definition}

Now, we introduce the definition of weakly relatively hyperbolic groups.

\begin{definition}
The group $G$ is said to be \emph{weakly relatively hyperbolic} with respect to $\mathscr{H}$ if the coned-off Cayley graph $\hat{\Gamma}$ of $G$ with respect to $\mathscr{H}$ forms a negatively curved metric space.    
\end{definition}

To define relatively hyperbolicity, we introduce another geometrical property known as \emph{bounded coset penetration}. First, we define relative (quasi) geodesics.

Given a fixed generating set $S$ of $\Gamma$, we choose a set of words $\{y_i^j\}$ that generate all $H_i\in \mathscr{H}$ for each $i$. Given a path $w$ in $\Gamma$, we construct a path $\hat{w}$ as follows: we search for subwords $y_i^j$ in $w$ and select only the leftmost word in case of overlaps. For each subword $z\in H_i$ as a product of $y_i^j$, let $z$ go from $g$ to $g\bar{z}$ in $\Gamma$. We then replace the path given by $z$ with an edge from the vertex $g$ to the cone point $v(gH_i)$, and an edge from $v(gH_i)$ to the vertex $g\bar{z}$. performing 
this operation for all such subwords and all $i$, we denote the new path in $\hat{\Gamma}$ by $\bar{w}$.

\begin{definition}
If $\hat{w}$ is a geodesic in $\hat{\Gamma}$, we call $w$ a relative geodesic in $\Gamma$. If $\hat{w}$ is a $P$-quasi-geodesic in $\hat{\Gamma}$, we call $w$ a relative $P$-quasi-geodesic in $\Gamma$. A path $w$ in $\Gamma$ (or $\hat{w}$ in $\hat{\Gamma}$) is said to be a path without backtracking if, for every coset the $\hat{w}$ penetrates, $\hat{w}$ never returns to $gH$ after leaving $gH$.
\end{definition}

\begin{definition}
The pair $(\Gamma,\mathscr{H})$ is said to satisfy the bounded coset penetration property (or BCP property for short) if, for every $P\geq 1$, there exists a constant $c$ such that if $u$ and $v$ are relative $P$-quasi-geodesics without backtracking with $d_{\Gamma}(\bar{u},\bar{v})\leq 1$, then the following conditions hold:

1. If $u$ penetrates a coset $gH_i$ but $v$ does not penetrate $gH_i$, then $u$ travels a $\Gamma$-distance of at most $c$ in $gH_I$.

2. If $u$ and $v$ penetrate the same coset $gH_i$, then the vertices in $\Gamma$ at which $u$, $v$ first enter $gH_i$ lie at a $\Gamma$-distance of at most $c$ from each other, and similarly for the vertices at which $u$, $v$ last exit $gH_i$.
\end{definition}

\begin{definition}
 Let $\mathscr{H}$ be a set of subgroups of a group $G$. $G$ is said to be relatively hyperbolic with respect to $\mathscr{H}$ if $G$ is weakly hyperbolic relative to $H$ and $(G,\mathscr{H})$ satisfies the BCP property.
\end{definition}

Although relative hyperbolicity is not easy to determine from its definition, the following theorem provides a powerful characterization. For a more comprehensive 
result in a different language, refer to \cite[Theorem 7.11]{BD}. 

\begin{definition}
A subgroup $H$ of a group $G$ is called \emph{malnormal} if for any $x\notin H$, $x^{-1}Hx$ and $H$ have trivial intersection.    
\end{definition}

\begin{theorem}\label{convex}
Let $G$ be a torsion-free hyperbolic group. If $\mathscr{H}$ is a finite family of quasiconvex malnormal subgroups of $G$, and the intersection of conjugates of any distinct pair of subgroups in $\mathscr{H}$ is trivial, then $G$ is relatively hyperbolic with respect to $\mathscr{H}$.
\end{theorem}

\subsection{(Relative) Dehn functions}

For a finitely presented group $G=\langle X; R\rangle$, a word $w\in F(X)$ represents the identity element in G if and only if it can be expressed as a finite product of conjugates of elements in $R$. Among all such representations, we consider the smallest integer $n$ such that $w$ can be represented by a product of conjugates of $n$ elements in $R$. We refer to $n$ as the \emph{area} of $w$. 

Van Kampen diagrams are standard tools for studying isoperimetric functions in groups. For detailed explanations, refer to \cite{M1}. The idea behind a van Kampen diagram is to 
associate a simply connected planar 2-complex with a word $w=_G1$,  where the 2-cells are labeled by relators and the boundary is labeled by $w$. 

The \emph{area} of a van Kampen diagram corresponds to the number of 2-cells it contains. The area of $w$, as defined earlier, represents the minimum area of a van Kampen diagram 
with a boundary labeled by $w$.
 
For a finitely presented group $G=\langle X;R\rangle$, the \emph{Dehn function} is defined as: $$D(n)=\max\{{\mathrm{Area}}(w)|w=_G1,|w|\leq n\}.$$

It is worth noting that different presentations of the same group yield equivalent Dehn 
functions \cite{M1}. Two functions $f$ and $g$ are considered equivalent if $f\preceq g$ and $g\preceq f$, where $\preceq$ indicates the existence of a constant $C>0$ such that: $$f(n)\leq Cg(Cn+C)+Cn+C.$$

The Dehn function of a group is defined as the equivalence class of functions. Relative 
Dehn functions can also be defined; for detailed discussions, refer to \cite{O}.

Let $\mathscr{H}$ be a set of subgroups of $G$. If $G$ is generated by $\mathscr{H}$ and a finite set $\{x_1,\dots,x_m\}$, we say that $G$ is \emph{finitely generated with respect to $\mathscr{H}$}, and $G$ can be presented as
$$G=\langle \mathscr{H}, x_1, \dots, x_m;\mathscr{R}\rangle,$$
where $\mathscr{R}$ is the set of relators needed to incorporate $x_1,\dots,x_m$ with $\mathscr{H}$. This presentation is referred to as \emph{the relative presentation of $G$ with respect to $\mathscr{H}$}. If $\mathscr{R}$ is also a finite set, $G$ is said to be \emph{finitely presented with respect to $\mathscr{H}$}.

Given such a relative presentation, a word is expressed as a product of words in $H_i\in \mathscr{H}$ and $\{x_1,\dots,x_m\}$. The \emph{relative length of a word $z$ with respect to $\mathscr{H}$}, denoted by $|z|_{\mathscr{H}}$, is defined as the number of $x_i$ and words in $\mathscr{H}$. In particular, any word in $H_i\in\mathscr{H}$ has length $1$.

A word $z=_G1$ can be expressed as a finite product of elements in $H_i\in\mathscr{H}$ and conjugates of elements in $\mathscr{R}$. The \emph{area of $z$ with respect to $\mathscr{H}$}, denoted by ${\mathrm{Area}}_{\mathscr{H}}(z)$, is defined as the minimal number of elements in $\mathscr{R}$ required for such a product. The \emph{relative Dehn function of this presentation} is defined as
$$D_{\mathscr{H}}(n)=\max\{{\mathrm{Area}}_{\mathscr{H}}(z)|z=_G1,|z|_{\mathscr{H}}\leq n\}.$$

The following theorem establishes a powerful connection between relative Dehn 
functions and relative hyperbolicity. For further details, refer to \cite[Theorem 1.5]{O}.

\begin{theorem}\label{relative}
Let $G$ be a finitely generated group and $\mathscr{H}$ be a finite collection of subgroups. Then the following conditions are equivalent:
\begin{enumerate}
\item
G is finitely presented with respect to $\mathscr{H}$, and the corresponding relative Dehn function is linear.
\item
G is hyperbolic relative to $\mathscr{H}$.
\end{enumerate}
\end{theorem}

\subsection{Graphs of groups and Dehn twists}\label{Dehn Twist}

 In this section, we introduce graph of groups and Dehn twists. See \cite[Chapter 5]{S} and \cite[Section 5,6]{CM} for more details.

Let $\Gamma$ be a connected oriented graph. A graph of groups over $\Gamma$ is given by the following:

(1) Each vertex $v\in \Gamma$ is associated with a group $G_v$, called the vertex group.

(2) Each edge $e\in \Gamma$ is associated with a group $G_e$, called the edge group.

(3) For each edge $e$, there are monomorphisms $\alpha_{i_e}: G_e\rightarrow G_{i(e)}$ and $\alpha_{t_e}: G_e\rightarrow G_{t(e)}$, where $i(e)$ and $t(e)$ are the initial and terminal vertices of $e$.

Label all edges by $e_0,\dots,e_m$, and let $t_k$ be a stable letter assigned to $e_k$. Let $T$ be a maximal tree in $\Gamma$. Then there is a unique group, called the fundamental group associated to $\Gamma$, denoted by $\pi_1(\Gamma, T)$, defined as follows.

It is generated by all vertex groups together with all stable letters.
For each $k$, if $e_k\in T$, then $t_k=1$ and $\alpha_{i_{e_k}}(a)=\alpha_{t_{e_k}}(a)$ for all $a\in G_{e_k}$. If $e_k\notin T$, then $t^{-1}_k\alpha_{i_{e_k}}(a)t_k=\alpha_{t_{e_k}}(a)$ for all $a\in G_{e_k}$.

The group elements can also be defined by loops on $\Gamma$. Fix a basepoint $v_0$, let $c=e_{i_1}\dots e_{i_n}$ be a loop on graph based on $v_0$. Then a \emph{word of type $c$} is a sequence $x_0t_{i_1}\dots x_{n-1}t_{i_n}x_n$, where $x_0\in G(v_0)$ and $x_j\in G_{t(e_{i_j})}$. These words form a subgroup of the free group generated by generating sets of vertex groups and stable letters. We denote this subgroup by  $F(\Gamma, v_0)$ and by abusing language, we call these words loops, too. We can reduce a loop by changing the segment $t^{-1}_k\alpha_{i_{e_k}}(a)t_k$ to $\alpha_{t_{e_k}}(a)$, or $t_k\alpha_{t_{e_k}}(a)t^{-1}_k$ to $\alpha_{i_{e_k}}(a)$, thus defining an equivalence relation on the set of loops on $(\Gamma, v_0)$. The set of equivalence classes under concatenations forms a group, which we denote by $\pi_1(\Gamma, v_0)$. 

There is a canonical projection: $p: \pi_1(\Gamma, v_0)\rightarrow \pi_1(\Gamma,T)$, by deleting those stable letters corresponding to edges in $T$. This projection is actually an isomorphism. See \cite[Chapter 5]{S} for details.

In this paper, all the edge groups are infinite cyclic groups $\mathbb Z=\langle s\rangle$. Then the monomorphisms associated to the edge $e$ are determined by $\alpha_{i_e}(s)$ and $\alpha_{t_e}(s)$. In this case, we pick $u_e=\alpha_{i_e}(s)$ and $w_e=\alpha_{t_e}(s)$ for each edge $e$, and we call them \emph{initial} and \emph{terminal} elements of the edge $e$. Up to conjugacy, they determine the same element in $G$, which we call the edge element.

For example, the amalgamation can be expressed as a tree of groups with two vertices and one edge connecting them. The HNN-extension can be expressed as a graph of groups with one vertex and one edge. 

For graph of groups with cyclic edge groups, we can define a basic Dehn twist associated to each edge.

\begin{definition}
For each edge $e_k$, the basic Dehn twist $\varphi_{k}$ is defined as follows:
\begin{enumerate}
\item 
Suppose that $e_k$ is in the maximal tree $T$, and it separates the graph $\Gamma$ into two components, say  $\Gamma_i$ and $\Gamma_t$, which contains $i({e_k})$ and $t({e_k})$, respectively. Let the initial and terminal elements of $e_k$ be $u_k$ and $w_k$, respectively. Then $\varphi_k$ is the automorphism that restricts to a conjugation by $u_k$ on $\Gamma_i$ and restricts to the identity on $\Gamma_t$.
\item 
Supppose that $e_k$ is in the maximal tree $T$ but does not separate the graph $\Gamma$. Since it separates the tree $T$ into two components, it defines an automprphism on $G(T)$ as above. This also induces a map on stable letters not in $T$: if $e_j$ connects $T_i$ and $T_t$ in the same direction as $e_k$ does, then $\varphi_k$ maps the corresponding stable letter $t_j$ to $u_k^{-1}t_j$; if $e_j$ connects $T_i$ and $T_t$ in the reverse direction as $e_k$ does, then $\varphi_k$ maps $t_j$ to $t_ju_k$.
\item 
Suppose that $e_k$ is not in $T$, then
$\varphi_k$ maps $t_k$ to $u_kt_k$ where $u_k$ is the initial element, and is the identity on other stable letters and all vertex groups.  
\end{enumerate}
\end{definition}

If $\bar{e}$ is the reverse edge of $e$, we replace $w_e$ and $u_e$ by $w_e^{-1}$ and $u_e^{-1}$ respectively on $\bar{e}$. In this way, we have $\varphi_{\bar{e}}(x)=u_e\varphi_e(x)u_e^{-1}$. So, for reverse edges, their corresponding Dehn twists differ by a conjugation.

Let $e$, $e'$ be different edges in $\Gamma$, then direct calculation shows that $\varphi_{e'}\varphi_e$ and $\varphi_{e}\varphi_{e'}$ also differ by a conjugation.

In general, we call $\varphi$ a \emph{Dehn twist} if it is a composition of some basic Dehn twists. According to previous discussions, all Dehn twists form an abelian subgroup of $\mathrm{Out}(G)$. This subgroup is, in some sense, generated by all edges regardless of their orientations. Based on this, we can also talk about the Dehn twist on an undirected graph of groups, but we still use orientation to indicate the exact Dehn twists in this paper.

\subsection{Mapping torus}
 Let us recall the definition of the mapping torus. For any automorphism $\varphi$ of a group $G$, the algebraic mapping torus is defined by 
 $$M_{\varphi}=\langle G,t;t^{-1}at=\varphi (a), a\in G\rangle.$$ Note that $G$ and $\langle t\rangle$ are embedded subgroups in the mapping torus.

In 2010, Bridson and Groves \cite{bridson2010quadratic} proved the following theorem about mapping tori of free groups. 

\begin{theorem}\label{3}
If $F$ is a finitely generated free group and $\varphi$ is an automorphism of $F$, then the mapping torus $M_{\varphi}$ satisfies a quadratic isoperimetric inequality.
\end{theorem}

\section{Operations on van Kampen diagrams}\label{oper: Van-Kampen}
In this section, we will discuss the geometric aspects of mapping tori and explore van Kampen diagrams and some operations on them. An operation refers to a procedure that changes a van Kampen diagram to a new one while keeping the boundary label unchanged.

Without loss of generality, we assume that all relators are reduced words, and the words under consideration are also reduced. This ensures that the boundary label of a van Kampen diagram is always a reduced word.

\subsection{Foldings on diagrams}
Similar to geometry, a \emph{path} in a van Kampen diagram is a continuous map $f$ from $[0, N]$ to the one-skeleton of the diagram. We require the path to be combinatorial, meaning that it is a homeomorphism from $[k, k + 1]$ to a 1-cell for each $k \in {0, \ldots, N - 1}$. Concatenation of paths is defined in the usual way. A path is called a \emph{loop} if $f(0) = f(N)$. A path is called \emph{tight} if it does not contain a backtracking, which means that $f([k - 1, k])$ and $f([k, k + 1])$ are not inverse edges for any $k$. The \emph{label} of a path is the label of its image, and it is reduced if the label is a reduced word. In this paper, we do not distinguish between a path and its label.

If a path has backtracking, we can remove all backtracking from the path, resulting in a tight path. Note that we do not change the diagram itself, but only modify the path. A tight path may not be reduced, in which case we need to change the diagram in order to reduce it.

Stallings foldings are applied to eliminate unreduced labels in a graph. However for two-dimensional complexes, like planar van-Kampen diagrams, we cannot guarantee that all paths are reduced. Nonetheless, we have some operations for folding certain edges based on our needs.

Here we introduce two kinds of foldings on a van Kampen diagram: directed folding and rotated folding.

\subsubsection{Directed Folding}

Suppose there is a path passing through three different vertices $A$, $B$, and $C$ consecutively, where $AB = a$ and $BC = a^{-1}$. We call such a segment $ABC$ an \emph{unreduced segment}. Note that it is not a backtrack. To fold this part, we need to consider whether there are other edges with $B$ as a vertex on both sides of $ABC$.

\begin{figure}
	\includegraphics[width=100mm]{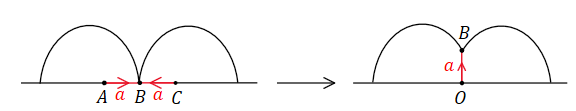}
	\caption{Direct folding}
	\label{fold1}
\end{figure}

If there are no such edges on one side of the path, we can fold it on that side. This folding is called direct folding. Figure \ref{fold1} illustrates this folding. After this folding, the unreduced segment turns into a backtrack, and we can remove it from the path.

Note that if there are no such edges on the other side either, it results in a vertex of valence 1 on this folded edge, which we call a \emph{branch}. We can remove this branch by deleting the edge but keeping the vertex that is also on other edges.

\begin{figure}
	\includegraphics[width=100mm]{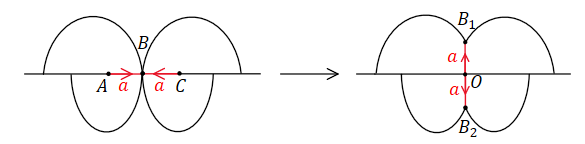}
	\caption{Rotated folding}
	\label{fold2}
\end{figure}

\subsubsection{Rotated Folding}

Suppose there are edges with vertex $B$ on both sides of the segment. Locally, $ABC$ divides the diagram into two parts, and it lies on the boundary of the closure of each part. In \emph{rotated folding}, we perform direct folding on each part separately and then glue them back along the new boundary. After this folding, vertices $A$ and $C$ are merged into one point, and we can remove the unreduced segment from the path.  Figure \ref{fold2} illustrates this folding.

If another segment of this path also passes through $B$ from one side of $ABC$ to the other side, we call it a \emph{crossing} at $B$. In this case, we need to insert a segment $a^{-1}a$ into this segment after the folding, so we cannot fully reduce it on this path. However, if a path has no crossings, the following lemma holds and the proof is straightforward.

\begin{lemma}\label{path}
Let $D$ be a van-Kampen diagram and  $\alpha$ be a path on $D$ without crossings. Then there is an operation on $D$, called reducing $\alpha$, such that $\alpha$ is transformed to a reduced path (the reduced form of $\alpha$), and the area of the diagram remains unchanged after this operation.
\end{lemma}

\subsection{The induced map on diagrams}\label{induced}
A group can always be presented as $G = \langle X; R \rangle = F / \llangle R \rrangle$, where $F$ is a free group generated by $X$ and $\llangle R \rrangle$ denotes the normal closure of $R$. If $\Phi$ is an automorphism of $F$ such that $\Phi(\llangle R \rrangle)=\llangle R \rrangle$, it induces an automorphism $\varphi$ on the quotient group $F / \llangle R \rrangle$. If an automorphism $\varphi$ of $G$ comes from an automorphism $\Phi$ of $F$ in this way, we say that $\varphi$ is \emph{induced from the free group automorphism $\Phi$}.

Suppose an automorphism $\varphi$ of $G$ is induced from the free group automorphism $\Phi$. Given a van Kampen diagram $D$ of a word $z =_ G1$ in $G$, there is another van Kampen diagram $\Phi(D)$ as follows:
\begin{enumerate}
    \item First, consider a cell labeled by a relator $r \in R$. Replace each edge representing a generator with a segment representing its image under $\Phi$. At this stage, there may be some unreduced segments in the diagram.

    \item Perform direct foldings on each unreduced segment and delete any created branches. Then, fill the new loop with the minimal number of cells. This process is valid since $\Phi$ maps $\llangle R \rrangle$ to $\llangle R \rrangle$. The resulting diagram is a van Kampen diagram with its boundary labeled by the reduced form of $\Phi(r)$. If we fix the way we fill each loop, it defines a map from all cells of $G$ to van Kampen diagrams of $G$.
\end{enumerate}

Now, $D$ is composed of cells labeled by generators of $G$. Denote the cells by $c_1, \ldots, c_n$. Follow these steps:
\begin{enumerate}
    \item Replace each edge representing a generator with a segment representing its image. The boundary of each cell $c_i$ is replaced by a new loop, and we consider these loops.

    \item If there is an unreduced segment on any loop, perform direct foldings to reduce it and delete any branches if present. Repeat this process a finite number of times until all unreduced segments on all loops are eliminated.

    \item All loops are now labeled by a reduced word, which is the image of the boundary of the corresponding cell. Fill each loop in a fixed way.

    \item The above steps result in a new van Kampen diagram. If the boundary of this diagram is not reduced, perform direct foldings to reduce it.
\end{enumerate}

Finally, we obtain a van Kampen diagram with its boundary labeled by the reduced form of $\Phi(z)$, and we denote this diagram by $\Phi(D)$. Although $\Phi(D)$ is not unique and depends on the order in which we fold all unreduced segments and the filling of the loops, the area of $\Phi(D)$ is well-defined.

If $R =\{r_1, \ldots, r_m\}$ is finite, there are finitely many different cells, denoted by $C_1, \ldots, C_m$ respectively. Let $A_i$ denote the area of $\Phi(C_i)$. Let $A = \max A_i$. For any van Kampen diagram $D$, we have $\text{Area}(\Phi(D)) \leq A \cdot \text{Area}(D)$.

A special case arises when $A = 1$, which means the image of one cell is exactly another cell. In this case, the image of a van Kampen diagram has the same area, and we say that $\Phi$ is \emph{area-preserving}. Note that if $\Phi$ is area-preserving, $\Phi^l$ is also area-preserving for any integer $l$.

\subsection{Room moving}
In a van Kampen diagram of the mapping torus $M_\varphi = {G, t; t^{-1}at = \varphi(a), a \in G}$, the cells consist of those from relations in $G$ and those from the relations $t^{-1}at = \varphi(a)$. We refer to the cells of the first type as \emph{primitive cells} and the cells of the second type as \emph{t-cells}.

Within the diagrams of mapping tori, t-corridors have been extensively studied, see \cite{bridson2010quadratic}. A \emph{t-corridor} is essentially a series of t-cells connected along t-sides. Figure \ref{t-corridor} illustrates an example of a t-cell and a t-corridor.

In this paper, t-corridors are considered to be maximal, meaning they cannot be extended further from either side. When it is clear from the context, we may also refer to a t-corridor simply as a corridor.

\begin{figure}
\includegraphics[width=90mm]{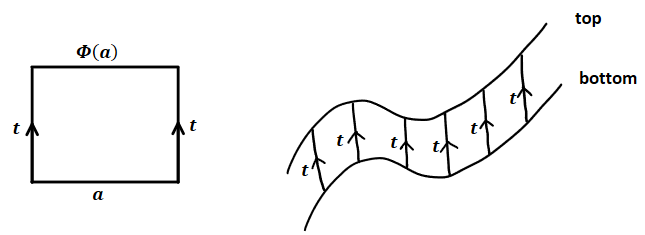}
\caption{}
\label{t-corridor}
\end{figure}

Depending on the structure of the t-cell, a t-corridor either starts and ends on the boundary or forms a ring, which we call a \emph{t-ring}. However, due to the following lemma, we only need to consider corridors of the first type.

\begin{lemma}
In a van Kampen diagram, there is an operation that removes each $t$-ring. Furthermore, if $\Phi$ is an area-preserving map, the number of primitive cells remains unchanged after this operation.
\end{lemma}
\begin{proof}
\begin{figure}
	\includegraphics[width=90mm]{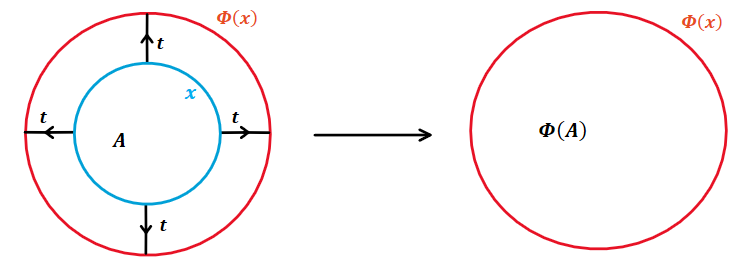}
	\caption{Remove a t-ring}
	\label{t-ring}
\end{figure}

The tops and bottoms of all $t$-rings do not have crossings with each other, allowing for simultaneous reduction. Hence we may assume that they are reduced.

Suppose the inner loop in a t-ring is labeled by $x$, and the enclosed subdiagram is denoted as $A$. The operation involves removing everything inside the outer loop, including $A$ and the t-ring, and refilling it with $\Phi(A)$. Figure \ref{t-ring} provides an illustration of this operation. If the t-sides of a t-ring are facing inward, the same result follows by replacing $\Phi$ with $\Phi^{-1}$.

If $\Phi$ is an area-preserving map, then $\Phi^{-1}$ is also area-preserving, see Section \ref{induced}. So the area of $\Phi(A)$ or $\Phi^{-1}(A)$ is the same as the area of $A$. Thus, the number of primitive cells remains unchanged after each operation.
\end{proof}

For the remainder of the paper, all corridors are assumed to start and end on the boundary. In this case, a corridor divides the diagram into two parts: one on the bottom side and one on the top side. This allows us to discuss the relative position of other cells or subdiagrams with respect to this corridor.

\begin{definition}
In a van Kampen diagram, a \emph{room} is the closure of one component after deleting all $t$-corridors. A $t$-corridor \emph{bounds a room} if they have a non-empty intersection.
\end{definition}
It is worth noting that if a t-corridor bounds a room, the intersection forms a connected path or a point.

Additionally, the paper introduces another operation called \emph{room moving}, which essentially allows for moving a room across a corridor. The following lemma describes this operation:

\begin{lemma}\label{7}
If a $t$-corridor bounds a room in a van Kampen diagram, there is an operation that moves all primitive cells within this room to the other side of the corridor. If $\Phi$ is an area-preserving map, the number of primitive cells in the diagram remains unchanged after the operation.
\end{lemma}

\begin{proof}

\begin{figure}
	\includegraphics[width=120mm]{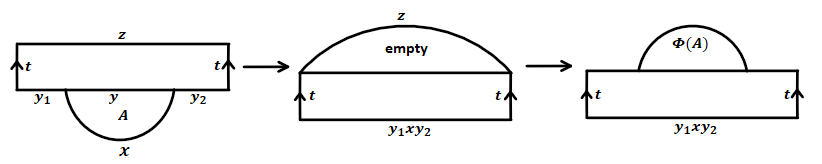}
	\caption{}
	\label{move1}
\end{figure}

\begin{figure}
	\includegraphics[width=120mm]{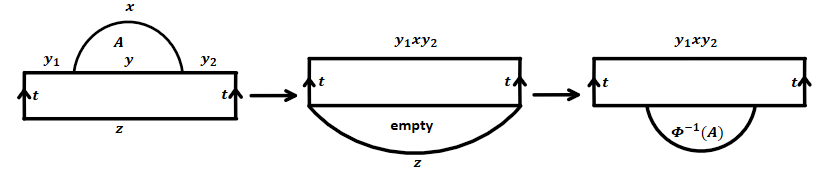}
	\caption{}
	\label{move2}
\end{figure}
In the case where the room $A$ intersects the corridor at a path $y$, and $A=xy^{-1}$ with $x$ being a segment of the boundary of $A$ that is not on the corridor, we can choose $x$ and $y$ such that the loop $xy^{-1}$ has no crossings. If $x$ is not a tight path, deleting all backtracks from $x$ forms a tight path. Denote this tight path by $x$ again. All cells in $A$ are between this tight path $x$ and $y$. Note that $x$ or $y$ may be a point. 

Consider the scenario where $y$ is at the bottom of the corridor. Assume the bottom of the corridor is $y_1yy_2$, and the top is denoted as $z$. We have $z=\Phi(y_1yy_2)=\Phi(y_1xy_2)$. Now, focus on the subdiagram that includes this corridor and the room $A$. By Lemma \ref{path}, we may assume the boundary of this subdiagram is already reduced, i.e., it consists of the reduced words $y_1xy_2$ and $z$.

To move $A$ to the other side of the corridor, we will remove other parts of the diagram, reconstruct this part, and then attach the other parts back in the same way.

First, construct a $t$-corridor with a bottom of $y_1xy_2$ and a top that is the reduced form of $\Phi(y_1xy_2)$. Next, attach a path $z$ to both ends of the top, forming a loop $\Phi(y_1xy_2)z^{-1}$ with its interior disjoint from the corridor. Note that $\Phi(y_1xy_2)z^{-1}=\Phi(xy^{-1})$, which means this loop corresponds to the image of the boundary of $A$.

Then, fold the two ends of $z$ together with the top of the corridor (if applicable) to create a new simplified loop. Fill this new loop with $\Phi(A)$. The resulting boundary is still $y_1xy_2tz^{-1}t^{-1}$, allowing us to attach the other parts of the diagram back.

The process is similar when $y$ is located at the top of the corridor. In this case, assume the top is $y_1yy_2$ and the bottom is $z$, giving $y_1yy_2=\Phi(z)$. To move $A$ to the other side, we need to construct a $t$-corridor with a bottom that is the reduced form of $\Phi^{-1}(y_1xy_2)$ and a top of $y_1xy_2$ (assuming it is already reduced). Attach a path $z$ to the bottom and fold both ends. Fill the resulting loop with $\Phi^{-1}(A)$.

These two cases are depicted in Figure \ref{move1} and Figure \ref{move2}, illustrating the steps involved in moving $A$ across the corridor.

By replacing $A$ with its image under $\Phi$ or $\Phi^{-1}$ on the opposite side of the corridor, we effectively move $A$ to the other side. Apart from $A$ and the corridor, the rest of the diagram remains unaffected. Therefore, if $\Phi$ is an area-preserving map, the total number of primitive cells in the diagram remains unchanged.
\end{proof}

\section{Operations on graphs of groups}\label{oper: graph of groups}
In this section, we introduce an operation on the graph of groups called sliding. This operation allows us to perform a simple form of general Dehn twists and then the operations discussed in the previous section could apply.

\subsection{Sliding of the graph} (See \cite{R})
Let $\Gamma$ be a graph of groups. From Bass-Serre theory, there exists a tree $T$ such that $\Gamma$ is the quotient of $G$ acting on $T$.
Consider a vertex $v_0$ with vertex group $G_0$ in $T$. For each edge $e_i$ adjacent to $v_0$, there is a corresponding edge element $u_i \in G_0$. If, for some $i$ and $j$, $u_i$ is a proper power of $u_j$, we have an operation on $\Gamma$ called sliding of $e_i$ along $e_j$.

Let $u_1=_{G_0}u_2^p$. Let the edge relation on $e_i$ be $w_i=u_i$, where $w_i\in G_i$, the vertex group of $v_i$, $i=1,2$. Then there are equations $w_1=_Gu_1=_{G_0}u_2^p=_{G}w_2^p$. 

\begin{figure}
	\includegraphics[width=100mm]{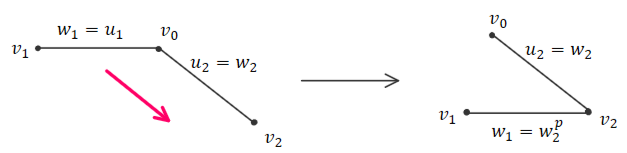}
	\caption{}
	\label{sliding}
\end{figure}

The sliding operation involves detaching $e_1$ from $v_0$ and then attaching it to $v_2$ at the same end. The new edge relation corresponding to this edge is $w_1 = w_2^{p}$. Figure \ref{sliding} illustrates this operation, which can be viewed as sliding the edge $e_1$ along $e_2$ to the other end of $e_2$, as indicated by the red arrow. 

Sliding on $T$ $G$-equivalently induces an operation on $\Gamma$, which we refer to as sliding. Denote the new graph of groups by $\Gamma'$. Note that $G(\Gamma')$ is generated by the same generators but changes the relation $w_1 = u_1$ to $w_1 = w_2^{p}$ from $G$. So it can be viewed as the same group under different presentations.

\begin{remark}\label{sliding1}
Consider only the subgraph shown in Figure \ref{sliding}. If $w_2$ is a proper power in $G_2$, say the root is $v\in G_2$, then the root of the edge element of $e_1$ is not in either of its two adjacent vertex groups but in $G_2$. This is fixed after the sliding, i.e., the root of each edge element is in one of its adjacent vertex groups. This is the type of graph of groups that we will investigate in the next section.
\end{remark}

In this section, we will prove the following lemma.

\begin{lemma}\label{D-sliding}
Let $D(\Gamma)$ denote the subgroup of $Out(G)$ generated by all Dehn twists in $\Gamma$. Then, for any $\varphi\in D(\Gamma)$, there exists $k$ such that $\varphi^k\in D(\Gamma')$.
\end{lemma}

\begin{proof}
Because of the commutativity of $D(\Gamma)$, it suffices to consider basic Dehn twists.

Let $\varphi_i$ and $\varphi'_i$ be basic Dehn twists corresponding to $e_i$ in $\Gamma$ and $\Gamma'$ respectively. If $i\neq 2$, they represent the same automorphism of $G$. So we only need to find $k$ such that $\varphi^k_2\in D(\Gamma')$.

We only cover the case when $e_1$ and $e_2$ is in the maximal tree of $\Gamma$, and the maximal tree of $\Gamma'$ is obtained by replacing $e_2$ to the new edge. The other case is easier and follows the same way.

When restricted to the maximal tree, removing $e_1$ and $e_2$ will result in three components. Denote them by $T_0$, $T_1$ and $T_2$ containing $v_0$, $v_1$ and $v_2$ respectively.
Let the direction of $e_2$ be from $v_0$ to $v_2$, then $\varphi_2$ is the map that restricts to a conjugation by $u_2$ on $G(T_0)$ and $G(T_1)$, and an identity on $G(T_2)$. On the other hand, $\varphi'_2$ is the map that restricts to a conjugation by $u_2$ on $G(T_0)$, and an identity on $G(T_1)$ and $G(T_2)$. Thus, they act differently on $G(T_1)$.

Now in $\Gamma'$, let the direction of $e_1$ be from $v_1$ to $v_2$, then $\varphi'_1$ is the map that restricts to a conjugation by $w_1$ on $G(T_1)$, and an identity on $G(T_0)$ and $G(T_2)$. Therefore, $\varphi_2^p$ is exactly the same isomorphism as $\varphi'_1(\varphi'_2)^p$ restricted on $G(T)\cong G(T')$. 

These two maps also induce the same map on stable letters corresponding to edges not on the maximal tree: if $e_i$ does not connect two different components, it is certainly true; if $e_i$ connects $T_0$ to $T_1$, they both act as conjugation by $u_1$ on $t_i$; If $e_i$ connects $T_0$ to $T_2$, they both map $t_i$ to $u_1^{-1}t_i$; if $e_i$ connects $v_1$ to $v_2$, they both map $t_i$ to $u_1^{-1}t_i$ as well.

As a result, $\varphi_2^p=\varphi'_1(\varphi'_2)^p\in D(\Gamma')$.

\end{proof}

\subsection{Full-rooted Graph of Groups}
For each edge $e$ in a graph of groups $\Gamma$ with cyclic edge groups, there are associated edge elements $u\in G_{i(e)}$ and $w\in G_{t(e)}$. If the root of $u$ in $G_{i(e)}$ is also a root in $G$, or the root of $w$ in $G_{t(e)}$ is also a root in $G$, we call $e$ \emph{a rooted edge}. If all edges in $\Gamma$ are rooted edges, we call $\Gamma$ \emph{a full-rooted graph of groups.}

From Remark \ref{sliding1}, we can convert an edge to a rooted edge by sliding. The next lemma shows how to do it to all edges, thus converting it to a full-rooted graph of groups. 

\begin{lemma}\label{full-rooted}
Let $\Gamma$ be a graph of groups with cyclic edge groups whose fundamental group is a torsion-free hyperbolic group $G$. Then there exists a sequence of slidings converting it to a full-rooted graph of groups $\Gamma'$.   
\end{lemma}

\begin{proof}
From Bass-Serre theory, there exists a tree $T$ such that $\Gamma$ is the quotient of $G$ acting on $T$.

An edge in $T$ is called rooted if it projects to a rooted edge in $\Gamma$. If all edges in $\Gamma$ are rooted, we are done. Otherwise, consider all edges that are not rooted in $T$. Let $e$ be a non-rooted edge in $T$ and let the edge elements be $w \in G_i(e)$ and $u \in G_{t(e)}$. Let $v$ be their root. Then $v$ acts on $T$ by an elliptic isometry. Hence it has a fixed point. Denote the fixed point set of $v$ as $\Omega_v$, and the geodesic connecting $e$ to $\Omega_v$ (including $e$) as $\phi(e)$. It is clear by construction that $\phi(ge) = g\phi(e)$ for all $g \in G$.

The path $\phi(e)$ is called maximal if it is not a proper subset of any other $\phi(e')$. The union of maximal paths in $\{\phi(e) \mid e \text{ is a non-rooted edge}\}$ is $G$-invariant and projects to finitely many paths in $\Gamma$.

Orient each maximal path so that it starts at a root. The slide along the first two edges in each maximal path projects to slides on $\Gamma$. Since it reduces the length of each maximal path, after finitely many steps, the resulting graph $T'$ is full-rooted. Therefore, $\Gamma' = T'/G$ is also a full-rooted graph of groups.
\end{proof}

\subsection{Lifting of the Dehn twist}

Let $\Gamma$ be a full-rooted graph of groups with cyclic edge groups. Fix a base point $v_0$ and the maximal tree $T$, we can label vertices in $T$ $v_0,v_1,\dots,v_m$ and all edges in $T$ $e_1,\dots,e_m$, ordered in such a way that we can construct $T$ by starting at $T_0=v_0$ and at $i$-th time, add new vertex $v_i$ together with edge $e_i$ connecting $v_i$ to $T_{i-1}$ to form a new tree $T_i$. We can also orient edges in $T$ so that the initial vertex of $e_i$ is $v_i$. Label the rest edges of $\Gamma$ by $e_{m+1},\cdot\cdot\cdot, e_n$.

For each $e_i$, let $v_i$ be the root of the edge element, which is in one of the adjacent vertex groups by our assumption. Then we can choose $u_i=v_i^{p_i}$ to be the edge element corresponding to this vertex, with $w_i$ to be the edge element corresponding to the other vertex. Note that $u_i$ may be in the initial or terminal vertex, but we assume to  always use $u_i=v_i^{p_i}$ to do conjugation or multiplication in the definition of the Dehn twist $\varphi_i$.

For an edge $e_i$ that is not in the maximal tree, let the direction be from the vertex containing $u_i$ to the vertex containing $w_i$. Thus, the basic Dehn twist $\varphi_i$ maps $t_i$ to $u_it_i$ and fixes other generators.

\begin{definition}
Let $(\Gamma, T)$ be labeled and oriented in the above sense. Then a Dehn twist of the form $\varphi=\varphi_n^{r_n}\dots\varphi_1^{r_1}$ is called a \emph{standard form}.
\end{definition}

Note that any Dehn twist on $\Gamma$ can be conjugated to a standard form. Since the mapping tori are isomorphic under conjugated automorphisms, it is enough to consider Dehn twists in the standard form.

\begin{lemma}
Let $\varphi=\varphi_n^{r_n}\dots\varphi_1^{r_1}$ be a Dehn twist in standard form. Then it can be lifted to a map $\bar{\varphi}: F(\Gamma, v_0)\rightarrow F(\Gamma, v_0)$ so that the diagram commute, in terms of words up to free reductions, for any power of $\varphi$ or $\varphi^{-1}$:
\[ \begin{tikzcd}
F(\Gamma,v_0) \arrow{r}{\bar{\varphi}^l} \arrow[swap]{d}{p} & F(\Gamma,v_0)\arrow{d}{p} \\%
\pi_1(\Gamma, T) \arrow{r}{\varphi^l}& \pi_1(\Gamma, T)
\end{tikzcd}
\]  
Furthermore, any word $x\in \pi_1(\Gamma, T)$ can be lifted to a word $L(x)$ in $F(\Gamma, v_0)$ so that $\varphi^l(x)=p\bar{\varphi}^l(L(x))$.
\end{lemma}

\begin{proof}
For a word $y=y_1\dots y_l\in F(\Gamma, v_0)$, where each $y_i$ is either an element in some vertex group or a stable letter, we define its image $\bar{\varphi}(y)$ as follows: if $y_i$ is a stable letter $t_j$, we replace it by $u_j^{r_j}t_j$; Otherwise, we keep it unchanged. Therefore, $\bar{\varphi}^l$ maps $t_j$ to $u_j^{lr_j}t_j$ and fixes all non-stable letters

We aim to use $\bar{\varphi}^l$ to analyze the image of $\varphi^l$. Hence we first consider how to lift a word $x\in \pi_1(\Gamma, T)$ to a word $L(x)$ in $F(\Gamma, v_0)$.

If $x$ is an element in some vertex group $G_v$, let $e=e_{l_k}\dots e_{l_1}$ be the unique geodesic in $T$ that connects $v$ to $v_0$. Then $L(x)=t^{-1}_{l_1}\dots t^{-1}_{l_k}xt_{l_k}\dots t_{l_1}$ corresponds to the lift of $x$. 

If $x=t_j$ is the stable letter corresponding to $e_j$ not in $T$, let $e_{j_k}\dots e_{j_1}$ be the geodesic in $T$ connecting the initial point of $e_j$ to $v_0$, and $e_{l_t}\dots e_{l_1}$ be the geodesic in $T$ connecting the terminal point of $e_j$ to $v_0$. Then $L(t_j)=t^{-1}_{j_1}\dots t^{-1}_{j_k}t_jt_{l_t}\dots t_{l_1}$ corresponds to the lift of $x$. 

In the above two cases, the lift is called a \emph{basic loop}. In general, if $x=x_1\dots x_n$ where each $x_i$ is either an element in a vertex group or a stable letter, we define $L(x)=L(x_1)\dots L(x_n)$ as the concatenation of basic loops after free reductions. Note that any elements in $F(\Gamma,v_0)$ can be expressed as a product of basic loops, it suffices to consider the commutativity on basic loops. 

Note that by our choice of order, $\varphi_i(u_j)=u_j$ if $i\geq j$.

Let $L(x)=t^{-1}_{l_1}\dots t^{-1}_{l_k}xt_{l_k}\dots t_{l_1}$ be a basic loop corresponding to $x\in G_{v_j}$. By definition, we have:
$$p\circ \bar{\varphi}^l(L(x))=u^{-lr_{l_1}}_{l_1}\dots u^{-lr_{l_k}}_{l_k}xu^{lr_k}_{l_k}\dots u^{lr_{l_1}}_{l_1}.$$

Now let us consider the image of $x$ under $\varphi$. For $i>m$, $\varphi_i$ fixes $x$. For $i\leq m$, $\varphi_i$ does not fix $x$ if and only of $e_i$ is on the geodesic that connects $v_j$ to $v_0$. In such a case, $\varphi_i$ acts on $x$ as a conjugation by $u_i$. Therefore, we have:
$$\varphi(x)=\varphi_{l_k}^{r_{l_k}}\dots\varphi_{l_1}^{r_{l_1}}(x)=u^{-r_{l_1}}_{l_1}\dots u^{-r_{l_k}}_{l_k}xu^{r_k}_{l_k}\dots u^{r_{l_1}}_{l_1}.$$

In the above expression of $\varphi(x)$, we treat the power of some $u_k$ as one term, together with the term $x$. If $e_i$ is not an edge on the geodesic that connects $v_j$ to $v_0$, then $\varphi_i$ fixes all terms. In other cases, $\varphi_i$ acts as a conjugation by $u_i$ on terms with subscripts greater than $i$, as well as on $x$. Therefore,
\begin{equation}
\begin{split}
\varphi^2(x)&=\varphi(u^{-r_{l_1}}_{l_1}\dots u^{-r_{l_k}}_{l_k}xu^{r_k}_{l_k}\dots u^{r_{l_1}}_{l_1})\\
&=\varphi_{l_k}^{r_{l_k}}\dots\varphi_{l_1}^{r_{l_1}}(u^{-r_{l_1}}_{l_1}\dots u^{-r_{l_k}}_{l_k}xu^{r_k}_{l_k}\dots u^{r_{l_1}}_{l_1})\\
&=\varphi_{l_k}^{r_{l_k}}\dots\varphi_{l_2}^{r_{l_2}}(u^{-2r_{l_1}}_{l_1}u^{-r_{l_2}}_{l_2}\dots u^{-r_{l_k}}_{l_k}xu^{r_k}_{l_k}\dots u^{r_{l_2}}_{l_2}u^{2r_{l_1}}_{l_1})\\
&=\ldots\\
&=u^{-2r_{l_1}}_{l_1}\dots u^{-2r_{l_k}}_{l_k}xu^{2r_k}_{l_k}\dots u^{2r_{l_1}}_{l_1}
\end{split}
\end{equation}

We continue this process, and finally, we obtain:
$$\varphi^l(x)=(\varphi_{l_k}^{r_{l_k}}\dots\varphi_{l_1}^{r_{l_1}})^l(x)=u^{-lr_{l_1}}_{l_1}\dots u^{-lr_{l_k}}_{l_k}xu^{lr_k}_{l_k}\dots u^{lr_{l_1}}_{l_1}.$$

Now let us consider the basic loop $L(t_j)=t^{-1}_{j_1}\dots t^{-1}_{j_k}t_jt_{l_t}\dots t_{l_1}$ corresponding to $t_j$. By definition, we have
$$p\circ \bar{\varphi}^l(L(x))=u^{-lr_{j_1}}_{j_1}\dots u^{-lr_{j_k}}_{j_k}u^{lr_j}_jt_ju^{lr_{l_t}}_{l_t}\dots u^{lr_{j_1}}_{j_1}.$$

For the image of $t_j$ under $\varphi$, if $e_i$ is not on either of the geodesics connecting two ends of $e_j$ to $v_0$, then $\varphi_i$ fixes $t_j$. For two geodesics connecting two ends of $e_j$ to $v_0$, there exists $h\geq0$ such that exactly the last $h$ edges are the same. That is to say, $j_i=l_i$ for $i\leq h$. Then for $i>h$, $\varphi_{j_i}(t_j)=u^{-1}_{j_i}t_j$ and $\varphi_{l_i}(t_j)=t_ju_{l_i}$; for $i\leq h$, $\varphi_{j_i}(t_j)=u^{-1}_{j_i}t_ju_{j_i}$.
Therefore, we have:
\begin{equation}
\begin{split}
\varphi(t_j)&=u^{-r_{j_1}}_{j_1}\dots u^{-r_{j_k}}_{j_k}u^{r_j}_jt_ju^{r_{l_t}}_{l_t}\dots u^{r_{l_{h+1}}}_{l_{h+1}}u^{r_{j_h}}_{j_h}\dots u^{r_{j_1}}_{j_1}\\
&=u^{-r_{j_1}}_{j_1}\dots u^{-r_{j_k}}_{j_k}u^{r_j}_jt_ju^{r_{l_t}}_{l_t}\dots u^{r_{l_1}}_{l_1}.
\end{split}
\end{equation}

Again, we treat the power of some $u_k$ as one term, together with the term $t_j$ in the above expression. If $e_i$ is not an edge on the geodesics that connects two ends of $e_j$ to $v_0$, then $\varphi_i$ fixes all terms. For $i\leq h$, $\varphi_{j_i}$ acts as a conjugation by $u_{j_i}$ on terms with subscripts greater than $j_i$, as well as on $t_j$. 

For $i>h$, $\varphi_{j_i}$ acts on $u^{-r_{j_{i+1}}}_{j_{i+1}}\dots u^{-r_{j_k}}_{j_k}u^{r_j}_j$ as conjugation by $u_{j_i}$, and acts on $t_j$ as a left multiply by $u^{-1}_{j_i}$. So as a result, $\varphi_{j_i}$ acts on the subword $u^{-r_{j_{i+1}}}_{j_{i+1}}\dots u^{-r_{j_k}}_{j_k}u^{r_j}_jt_j$  as a left multiply by $u^{-1}_{j_i}$ and fixes all other terms. Similarly, $\varphi_{l_i}$ acts on the subword $t_ju^{r_{l_t}}_{l_t}\dots u^{r_{l_{i+1}}}_{l_{i+1}}$ as a right multiply by $u_{l_i}$ and fixes all other terms.

By induction, we can derive:
$$\varphi^l(t_j)=u^{-lr_{j_1}}_{j_1}\dots u^{-lr_{j_k}}_{j_k}u^{lr_j}_jt_ju^{lr_{l_t}}_{l_t}\dots u^{lr_{l_1}}_{l_1}$$

In general, if $x=x_1\dots x_n$ where each $x_i$ is either an element in a vertex group or a stable letter,
\begin{equation}
\begin{split}
p\circ \bar{\varphi}^l(L(x))&=p\circ \bar{\varphi}^l(L(x_1)\dots L(x_n))\\
&=p\circ \bar{\varphi}^l(L(x_1))\dots p\circ \bar{\varphi}^l(L(x_n))\\
&=\varphi^l(x_1)\dots\varphi^l(x_n)\\
&=\varphi^l(x).
\end{split}
\end{equation}

Finally, $\varphi^{-1}$ can be obtained by replace $u_i$ by $u_i^{-1}$ in the definition of the map. So if $l$ is negative, it can be proved similarly by replacing $u_i$ with $u_i^{-1}$.
\end{proof}

\section{Proof of the Main Theorem}\label{proof}
In this section, we present the proof of Theorem~\ref{main}. We commence with the case of Dehn twists and establish the general case by reducing it to Dehn twists.
\subsection{Dehn twists}

\begin{lemma}\label{power}
If $u^m=w^n$ in a torsion-free hyperbolic group $G$, then there exists $v\in G$ such that both $u$ and $w$ are powers of $v$. 
\end{lemma}
\begin{proof}
    This follows from the following two facts:
   
    (1) According to \cite[page 462-463]{MA}, in word-hyperbolic groups, centralizers of elements of infinite order are virtually $\mathbb Z$. 
    
    (2)  A torsion-free group that is virtually $\mathbb Z$ is infinite cyclic.
\end{proof}

\begin{proposition}\label{graph}
Let $G$ be a torsion-free hyperbolic group represented by a full-rooted graph of groups. Then for any Dehn twist, the corresponding mapping torus satisfies a quadratic isoperimetric inequality.
\end{proposition}

\begin{proof}

We label and orient all edges as in the previous section, with edge elements $u_i$, $w_i$, and corresponding roots $v_i$. Since any Dehn twist can be conjugated to a standard Dehn twist, we only need to consider a standard Dehn twist $\varphi=\varphi_n^{r_n}\dots\varphi_1^{r_1}$.

Let $H_i$ be the subgroup generated by $v_i$. Since $G$ is torsion-free, each $H_i$ is a quasiconvex malnormal subgroup of $G$.

If $v_i^n=v^{-1}v^t_jv$, then $v_i^n=(v^{-1}v_jv)^t$. Since $v_i$ and $v_j$ are already roots, according to Lemma \ref{power}, they must be conjugate to each other. In this case, we choose only one of $H_i$ or $H_j$ and consider the family of subgroups $\mathscr{H}=\{H_i\}$. By Theorem \ref{convex}, $G$ is relatively hyperbolic with respect to $\mathscr{H}$.

Now let us consider a standard Dehn twist $\varphi=\varphi_n^{r_n}\dots\varphi_1^{r_1}$. Let $z=X_1t^{l_1}\dots X_nt^{l_k}\in M_{\varphi}$ be a word that equals to $1$ in $M_{\varphi}$, where $X_i$'s are elements in $G$.

The relations $t^{-1}at=\varphi(a)$ implies $t^{-1}a=\varphi(a)t^{-1}$ and $ta=\varphi^{-1}(a)t$, allowing us to move $t$ to the end of the expression and freely cancel them. Thus, we transform $z=X_1t^{l_1}\dots X_nt^{l_n}\in M_{\varphi}$ to $z'=X_1\varphi^{-l_1}(X_2)\dots\varphi^{-(l_1+\dots +l_{n-1})}(X_n)$.

Now, let us analyze each term $\varphi^{-(l_1+\dots +l_{i-1})}(X_i)$ using the lifted map introduced before: $\bar{\varphi}^{-(l_1+\dots +l_{i-1})}(X_i)$

Let $D$ be the diameter of $T$. The length of a basic loop $L(x)$ is bounded by $2D+|x|\leq (2D+1)|x|$. As concatenation, it is evident that the length of the lift of each $X_i$ satisfies $|L(X_i)|\leq (2D+1)|X_i|$.

According to the definition of $\bar{\varphi}$, $z'$ is obtained by inserting some terms in the from $(v_i^{p_i})^{-(l_1+\dots +l_{i-1})r_j}=v_i^{-(l_1+\dots +l_{i-1})r_jp_i}$, and the number of such insertions in no more than $(2D+1)|X_i|$. Therefore, the relative length of the final word with respect to $\mathscr{H}$ is no more than $(2D+2)|X_i|$.

Hence, the relative length of $z'$ with respect to $\mathscr{H}$ is no more than $(2D+2)\sum_{i=1}^n|X_i|\leq (2D+2)|z|$. Considering the relative Dehn function of $G$ with respect to $\mathscr{H}$, which is equal to its Dehn function since each $H_i$ is freely generated, there exists a constant $C$ such that the area of $z'$ is bounded by $C(2D+2)|z|$. 

The entire process $z\rightarrow z'\rightarrow 1$ yields a van Kampen diagram. During the transformation from $z$ to $z'$, we only utilized the relation $t^{-1}at=\varphi(a)$, indicating that the corresponding cells are all $t$-cells. Consequently, the total number of primitive cells in the diagram is bounded by $B|z|$, where $B=2C(D+1)$.

It is evident from the definition that each basic Dehn twist $\varphi_i$ is induced from a free group automorphism $\Phi_i$. Therefore, as a product, $\varphi$ is also induced from a free group automorphism $\Phi$.

Furthermore, each $\Phi_i$ preserves the area since it maps each cell to itself. Hence, $\Phi$ also preserves the area.

\begin{figure}
\includegraphics[width=110mm]{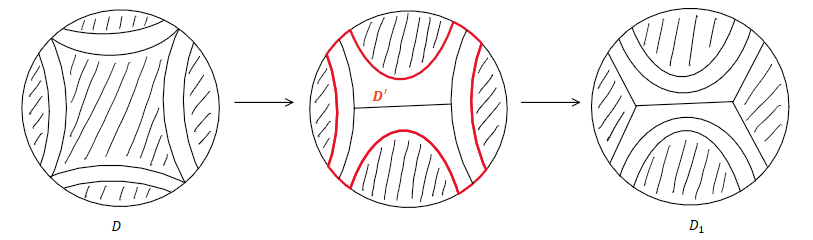}
\caption{}
\label{Transformation}
\end{figure}

Now it is ready to construct a new Van-Kampen diagram for $z$ from the diagram for $z\rightarrow z'\rightarrow 1$.

Initially, by employing the operation from Lemma \ref{7}, we can move all primitive cells towards the boundary so that the set of all $t$-corridors forms a connected, simply-connected subdiagram denoted as $D'$. Every edge on the boundary of $D'$ is either on the boundary of some primitive cell or on the boundary of the original diagram. According to Lemma \ref{7}, the number of primitive cells remains unchanged after these operations.

Suppose the maximal length of the boundary of each primitive cell is $m$. Then, the sum of the length of the boundaries of all primitive cells is bounded by  $Bm|z|$. Consequently, the length of the boundary of $D'$ is bounded by $Bm|z|+|z|=(Bm+1)|z|$.

The next step is refilling $D'$. 
It is noteworthy that $D'$ can be viewed as a van-Kampen diagram of the mapping torus of $\Phi$. According to Theorem \ref{3}, there exists a refilling of $D'$ with $t$-cells, the number of which is bounded by $M(Bm+1)^2|z|^2$ for some constant $M$ depending on $\Phi$. This yields a new diagram denoted as $D_1$. An example of this series of operations is illustrated in Figure \ref{Transformation}, where all primitive cells are shaded.
Since ${\rm{Area}}(z)$ is defined as the minimum area among all such diagrams, the bound for the area of  $D_1$ provides an upper bound for ${\rm{Area}}(z)$:
$${\rm{Area}}(z)\leq {\rm{Area}}(D_1)\leq M(Bm+1)^2|z|^2+B|z|.$$
\end{proof}

\subsection{General case}

We now proceed to prove the case of polynomial-growing automorphisms.

\begin{theorem}\label{polynomial-growing}
Let $G$ be a one-ended torsion-free hyperbolic group, and $\rho$ be a polynomial-growing automorphism of $G$. Then the corresponding mapping torus $M_{\rho}$ satisfies a quadratic isoperimetric inequality.
\end{theorem}
\begin{proof}
Levitt \cite{levitt2005automorphisms} has provided a classification, up to finite index, of all automorphisms of $G$, based on the JSJ decomposition. In the case of one-ended torsion-free hyperbolic groups, polynomially-growing automorphisms can be represented as Dehn twists up to power.
Let $\Gamma$ be a graph of groups representing $G$, and $D$ be a Dehn twist on $\Gamma_0$ that represents $\rho$, up to power.

If $\Gamma$ is not a full-rooted graph of groups, we can use Lemma \ref{full-rooted} to perform a sequence of slidings to convert it into a full-rooted graph of groups, denoted as $\Gamma'$. Additionally, Lemma \ref{D-sliding} guarantees the existence of an integer $k$ such that $D^k$ is an element of  $D(\Gamma')$, which we denote as $D'\in D(\Gamma)$.

Applying Proposition \ref{graph}, we conclude that $M_{D'}$ satisfies a quadratic isoperimetric inequality. Note that $\rho$ and $D'$ differ by a power; thus, their corresponding mapping tori have the same Dehn function. This completes the proof of the theorem. 
\end{proof}

We are now prepared to prove the main theorem.

\begin{manualtheorem}{\ref{main}}
Let $G$ be a one-ended torsion-free hyperbolic group, and let $\varphi$ be an isomorphism of $G$. Then the corresponding mapping torus $M_{\varphi}$ satisfies a quadratic isoperimetric inequality.
\end{manualtheorem}

\begin{proof}
Based on Theorem \ref{polynomial-growing}, the result follows from two important facts.

The first fact, established by Dahmani and Krishna \cite{dahmani2020relative}, states that the mapping torus of a torsion-free hyperbolic group is relatively hyperbolic with respect to the suspension of maximal polynomially-growing subgroups. These suspension subgroups are isomorphic to mapping tori of polynomial growth automorphisms, and according to Theorem~\ref{polynomial-growing}, they satisfy a quadratic isoperimetric inequality.

The second fact is a result of Farb \cite{farb1998relatively}, who demonstrated that if $G$ is relatively hyperbolic with respect to subgroups $\{H_i\}$, then the Dehn function of $G$ coincides with the maximum Dehn function among $\{H_i\}$.

By combining these facts, we conclude that the mapping torus $M_\varphi$ satisfies a quadratic isoperimetric inequality.
\end{proof}

\end{document}